\documentclass{article}
\usepackage{graphicx}

\title{Zonoid volumes are not log-submodular}
\author{
Ruben Skorupinski \thanks{EPFL, Switzerland, \textsf{ruben.skorupinski@epfl.ch}}
}
\date{}

\usepackage[switch,mathlines]{lineno}

\usepackage[margin=0.95in]{geometry}
\usepackage{inconsolata}
\usepackage{libertine}
\usepackage[T1]{fontenc}
\usepackage[utf8]{inputenc}
\usepackage{enumerate}
\usepackage[tbtags]{amsmath}
\allowdisplaybreaks
\usepackage{amssymb,amsthm,mathtools, bbm}
\usepackage{physics}
\usepackage{hyperref}
\usepackage[svgnames]{xcolor}
\usepackage[capitalise,nameinlink]{cleveref}
\definecolor{links}{RGB}{11, 85, 255}
\definecolor{cites}{RGB}{0, 200, 0}
\definecolor{urls}{RGB}{255, 116, 0}
\hypersetup{
    colorlinks=true,
    linkcolor=black,
    citecolor=black,
    urlcolor=black
}
\usepackage[
    style=numeric, 
    natbib=true,      
    doi=false,        
    maxbibnames=99,   
    maxcitenames=2,   
    url=false,        
    isbn=false,       
    backend=biber     
]{biblatex}
\usepackage{doi}
\usepackage{tikz}
\usetikzlibrary{positioning, arrows.meta, shapes.geometric}
\usepackage{pgfplots}
\pgfplotsset{compat=1.14}
\usepgfplotslibrary{fillbetween}
\usepackage{hyphenat}
\usepackage[font=footnotesize]{caption}
\usepackage{subcaption}
\usepackage{appendix}
\usepackage{thm-restate}
\usepackage{nicefrac}
\usepackage{tcolorbox}
\usepackage{bm}

\usepackage{todonotes}

\usepackage{algorithm}
\usepackage[noend]{algpseudocode}
        
\newcommand{\conv}{\textup{conv}}

\newcommand{\R}{\mathbb{R}}

\newcommand{\indep}{\perp\kern-5pt\perp}

\theoremstyle{theorem}

\newtheorem{theorem}{Theorem}[section]
\newtheorem*{theorem*}{Theorem}

\newtheorem{proposition}[theorem]{Proposition}
\newtheorem{corollary}[theorem]{Corollary}

\theoremstyle{remark}

\crefformat{equation}{#2(#1)#3}
\Crefformat{equation}{#2(#1)#3}

\begin{document}

\maketitle

\begin{abstract}   
    We disprove the conjecture that volume is log-submodular under Minkowski addition on the class of zonoids by exhibiting a four-dimensional zonotope $A$ and two segments $B$ and $C$ such that
    \begin{align*}
        |A||A+B+C|>|A+B||A+C|,
    \end{align*}
    where $|\cdot|$ denotes the volume. Consequently, several related local mixed-volume, local Loomis-Whitney, projection-volume ratio, and volume-to-surface-area conjectures also fail for zonoids. The zonotope in our counterexample is generated by a 2-modular matrix. We also provide a proof of the correctness of the conjecture in the case where $A+B+C$ is a unimodular zonotope as well as a characterization of the equality cases in this setting.
\end{abstract}

\section{Introduction}

A \emph{zonotope} $Z(V)$ generated by a matrix $V=(v_1,...,v_n) \in \R^{d \times n}$ is defined as the Minkowski sum of the line segments $\conv(0,v_i)= [0,v_i]$, that is, $Z(V) = \{\sum_i\lambda_i v_i: \lambda_i \in [0,1]\}$. Zonotopes and their Hausdorff limits, called \emph{zonoids}, have been studied extensively from both geometric and combinatorial viewpoints. A conjecture originally stated in \citet[Conjecture 4.16]{fradelizi2024sumset} proposes that for any three zonoids $A,B,C \subseteq \R^d$ it holds that 
\begin{align}\label{conj:log_sub}
    |A||A+B+C| \leq |A+B||A+C| 
\end{align}
It is shown in \cite{fradelizi2024volume} that the full conjecture \cref{conj:log_sub} is equivalent to the case above where $A$ is a zonotope and $B$ and $C$ are segments. In \cite{fradelizi2024sumset}, it is shown that this log-submodularity property of the volume holds if $A$ is an ellipsoid (in this case $B$ is even allowed to be any convex body) and that it fails for general convex bodies $A,B,C$. In a subsequent work \citep{fradelizi2024volume} the conjecture is discussed further and in particular is shown to hold for zonoids in dimensions 2 and 3. Furthermore, it is shown that the conjecture is equivalent to several other natural geometric inequalities which hence hold for 2- and 3-dimensional zonoids. To the best of our knowledge, the conjecture remained open for $d \geq 4$; it was explicitly treated as such in recent subsequent works \citep{fradelizi2026weighted, manui2025volume}.

\subsection{Results of this paper}

The main result (\Cref{thm:main_thm}) of this paper is a counterexample to the conjecture \cref{conj:log_sub} in $\R^4$. Indeed, taking $A= Z(V)$, $B= [0,w]$ and $C=[0,w']$, where
\begin{align*}
    V = \begin{pmatrix}
    1& 0& 0&  0&  0& 1\\
    0& 1& 0&  1&  0&-1\\
    0& 0& 1&  0&  1&-1\\
    0& 0& 0& -1& -1& 0\\
    \end{pmatrix}, 
    \quad w = \begin{pmatrix}
    1\\
    0\\
    0\\
    -1\\
    \end{pmatrix},
    \quad w'= \begin{pmatrix}
    0\\
    1\\
    -1\\
    0\\
    \end{pmatrix},
\end{align*}
results in $|A||A+B+C| >|A+B||A+C|$, disproving the conjecture for zonoids. Due to the other implications and equivalences shown in \citet{fradelizi2024volume} this counterexample also disproves additional conjectures for the class of zonoids including a local Loomis-Whitney inequality, a conjecture on the monotonicity of the projection volume of zonoids, and one on the monotonicity of volume-to-surface area. The implications are listed in more detail in \Cref{cor:main_cor}.
\par

A second result of the paper - and also the reason the example was found - is a proof of the conjecture \cref{conj:log_sub} when the sum of the participating zonotopes is unimodular. Unimodular zonotopes, i.e., zonotopes generated by a matrix $V$ whose $(d\times d)$-subdeterminants are all in $\{-1,0,1\}$, have already proved useful in a previous work \cite{eisenbrand2026nearly} where the structural properties of such zonotopes were used to prove a tight sparsification result for their generator sets. For a unimodular zonotope $Z(V)$ the volume can be written as $|Z(V)| = \det(VV^T)$ and the matrix determinant lemma can be applied to prove the log-submodularity of the volume and also characterize the equality cases. 

\paragraph{How the counterexample was found}
We first proved that the conjecture holds for zonotopes and two segments $w,w'$ if the matrix $(V|w|w')$ is unimodular. Next, we inspected 2-modular matrices. These matrices have the property that all their $(d\times d)$-sub-determinants lie in the set $\{-2,-1,0,1,2\}$ and were studied recently in \citet{oxley20222}. Chat-GPT was then used to search for a counterexample within the space of the 2-modular matrices which led to the discovery of the counterexample within a specific class of 2-modular matrices defined in \cite{oxley20222}. We remark that 2-modular matrices are highly structured (see e.g. \cite{artmann2017strongly, averkov2024maximal, dadush2026excluding, lee2023polynomial}), for example, a 2-modular matrix with $d$ rows can only have $O(d^2)$ unique columns. It thus seems plausible to computationally search the space of 2-modular matrices without relying on specific examples of the literature.

\section{Unimodular zonotopes}

In this section we provide a short proof that the volume of unimodular zonotopes is log-submodular when adding generators $w,w'$ that keep the zonotope unimodular. Unimodular zonotopes are a natural generalization of the class of graphical zonotopes\footnote{Given a connected graph $G=(V,E)$ define the generators $v_{\{i,j\}} = e_i-e_j$ for every edge $\{i,j\} \in E$. Note that these zonotopes are inherently $(d-1)$-dimensional since every generator of the zonotope is orthogonal to the all-1-vector.}, for which it is already known that their volume behaves in a log-submodular way. One can check that the $(d-1)$-dimensional volume of graphical zonotopes is exactly $\sqrt{d}$ times the number of spanning trees in the underlying graph, and the fact that this quantity behaves log-submodularly as a function of the edge set was shown in \citep{khosoussi2019reliable}.
\par
The main reason that Conjecture \cref{conj:log_sub} holds in the class of unimodular zonotopes comes down to the sub-determinants lying in $\{-1,0,1\}$. Because of this, the absolute value of the determinant of a submatrix $V_B = (v_i)_{i\in B}$ is equal to the square of the determinant. With this at hand, we can write the volume of a unimodular zonotope as the sum of squares of sub-determinants which behaves more benignly than the usual sum of absolute values. The ideas of the proof were inspired by \cite{eisenbrand2026nearly}.

\begin{proposition} \label{prop:unimodular_zonos}
    Let $Z$ be a zonotope generated by a rank-$d$ matrix $V \in \R^{d \times n}$ and let $w$, $w'$ be additional generators such that the matrix $(V|w|w')$ is unimodular. Then $|Z||Z+[0,w]+[0,w']| \leq |Z+[0,w]||Z+[0,w']|$. Moreover, equality holds if and only if $w\perp (VV^T)^{-1}w'$.
\end{proposition}

\begin{proof}
    Since all $(d\times d)$-sub-determinants $\det(V_B)$ of the unimodular matrix $V=(v_1,...,v_n)$ are in $\{-1,0,1\}$ the volume of the zonotope generated by $V$ is given by 
    \begin{align*}
        |Z| = \sum_{|B|=d} | \det(V_B)|  =  \sum_{|B|=d}\det(V_B)^2
    \end{align*} 
    By Cauchy-Binet we can thus write $|Z| = \det({\sum_i v_iv_i^T}) = \det(VV^T)$ and therefore, using the matrix-determinant lemma, we deduce that for an additional generator $w$ we get
    \begin{align*}
        |Z(V|w)|= \det\left( \sum_{i \in [n]}v_i v_i^T + ww^T \right) = \det\left( \sum_{i \in [n]}v_i v_i^T\right) \left(1 + w^T \left(\sum_{i \in [n]}v_i v_i^T\right)^{-1} w  \right).
    \end{align*}
    Note that the final expression is $ |Z(V)|(1+w^T(VV^T)^{-1}w)$ and hence the multiplicative increase in volume from adding the generator $w$ to $Z(V)$ is exactly the factor $1+w^T(VV^T)^{-1}w$. Similarly, when adding the same generator $w$ to the zonotope $Z(V|w')$ the relative increase is given by $1 + w^T(VV^T + w'w'^T)^{-1}w$. Finally, applying the Sherman–Morrison formula to $VV^T + w'w'^T$ and temporarily defining $M = VV^T$ for readability yields
    
    \begin{align*}
        &w^T(M + w'w'^T)^{-1}w \leq w^TM^{-1}w  \iff w^T\left(M^{-1}- \frac{M^{-1}w'w'^TM^{-1}}{1 + w'^TM^{-1}w'}\right)w \leq w^TM^{-1}w\\
        &\iff w^T\frac{M^{-1}w'w'^TM^{-1}}{1 + w'^TM^{-1}w'}w \geq 0 \iff (w'^TM^{-1}w)^2 \geq 0.
    \end{align*}

    Since the last inequality is true, with equality if and only if $w'^TM^{-1}w = 0$, the result follows.
\end{proof}

\section{2-modular zonotopes}


We will now present the 4-dimensional counterexample that rules out general log-submodularity of zonoid volumes in dimensions $d \geq 4$. For this we will make use of a class of matrices defined in the paper ``2-modular matrices'' by Oxley and Walsh \cite{oxley20222}. Staying close to their notation, we define $I_d = (e_1,...,e_d)$ to be the $d$-dimensional identity matrix, $D_d$ the matrix consisting of the columns $d_{ij} = e_i-e_j$ for every $1 \leq i <j \leq d$ and $h = d_{12}-e_3$. The matrix we will consider is a submatrix of $H_d := (I_d|D_d|h)$. Note that $(I_d|D_d)$ is totally unimodular and hence, using the multilinearity of the determinant to split up $h$ into $e_3$ and $d_{12}$ we find that all $(d \times d)$ sub-determinants of $H_d$ lie in the set $\{-2,-1,0,1,2\}$, i.e., $H_d$ is 2-modular.

\par

Now, consider the submatrix $V$ of $H_4$ given by 
\begin{align*}
    V = (e_1,e_2,e_3, d_{24}, d_{34},h)= \begin{pmatrix}
    1& 0& 0&  0&  0& 1\\
    0& 1& 0&  1&  0&-1\\
    0& 0& 1&  0&  1&-1\\
    0& 0& 0& -1& -1& 0\\
    \end{pmatrix}, 
    \quad w =d_{14} = \begin{pmatrix}
    1\\
    0\\
    0\\
    -1\\
    \end{pmatrix},
    \quad w' =d_{23} = \begin{pmatrix}
    0\\
    1\\
    -1\\
    0\\
    \end{pmatrix}.
\end{align*}

It is worth observing that if we removed the special column $h$ from the matrix $V$ to get $\Tilde{V}$, we would be back in the unimodular case. Moreover, one can check that $w^T(\Tilde{V}\Tilde{V}^T)^{-1}w' = 0$ and hence we recover the equality case of \Cref{prop:unimodular_zonos}, that is, $|Z(\Tilde{V})||Z(\Tilde{V}|w|w')| = |Z(\Tilde{V}|w)||Z(\Tilde{V}|w')|$. Appending the vector $h$ to all matrices increases the left-hand side of the equation more than it increases the right-hand side, and hence, computing the volumes of the four zonotopes yields

\begin{align*}
    |Z(V)||Z({V}|w|w')| = 14 \cdot 56 = 784 > 780 = 30 \cdot 26 = |Z({V}|w)||Z({V}|w')|.
\end{align*}

For completeness, let $N_j(M)$ denote the number of $(4 \times 4)$-sub-determinants of a matrix $M$ whose absolute value is $j$ for $j \in \{0,1,2\}$. The table below gives the counts of the three types of sub-determinants leading to the final volume inequality.

\begin{align*}
    \begin{array}{c|ccc|c}
    M
    & N_0(M)
    & N_1(M)
    & N_2(M)
    & |Z(M)|=N_1(M)+2N_2(M)
    \\ \hline
    V
    & 2  & 12 & 1 & 14
    \\
    (V|w)
    & 6  & 28 & 1 & 30
    \\
    (V|w')
    & 12 & 20 & 3 & 26
    \\
    (V|w| w')
    & 18 & 48 & 4 & 56
    \end{array}
\end{align*}
\\

The 4-dimensional counterexample can easily be lifted to work in any dimension $d \geq 4$. Indeed, by considering $\Hat{V}=\text{diag}(V,I_k)$ and $\Hat{w}=(w,0),\Hat{w}' =(w',0)$ each of the resulting zonotopes is the Cartesian product of the corresponding 4-dimensional zonotope with a $k$-dimensional cube $[0,1]^k$, while the combined generator matrix remains 2-modular. With this we therefore get the main theorem.

\begin{theorem} \label{thm:main_thm}
    For every $d \geq 4$ there exists a 2-modular matrix $(V|w|w')$ such that 
    \begin{align*}
        |Z(V)||Z({V}|w|w') > |Z({V}|w)||Z({V}|w')|.
    \end{align*}
\end{theorem}

Due to the equivalences and implications shown in \citet[Theorem 3.6, 3.9, and 5.3]{fradelizi2024volume} the following conjectures also fail within the class of zonoids.

\begin{corollary}\label{cor:main_cor}
    Denote by $|\cdot |_k$ the $k$-dimensional volume, by $P_{L}A$ the orthogonal projection of a $A$ onto the linear subspace $L$, by $\partial A$ the boundary of $A$ and by $V(\cdot)$ the mixed volume as defined in \citep{fradelizi2024volume}, where $A[j]$ denotes $j$ repeated copies of $A$. Then, each of the following conjectures is \emph{false} in every dimension $d\geq 4$.
    \begin{itemize}
        \item A local mixed-volume inequality: Let $A, B, C \subseteq \R^d$ be zonoids, then
           \[
            |A|\,V(A[d-2],B,C)
            \leq
            \frac{d}{d-1}
            V(A[d-1],B)\,V(A[d-1],C).
            \]
        \item A local Loomis-Whitney inequality: Let $A \subseteq \R^d$ be a zonoid and $u,v \in S^{d-1}$ linearly independent directions, then
            \[
            |A|\,
            \bigl|P_{\operatorname{span}\{u,v\}^{\perp}}A\bigr|_{d-2}
            \sqrt{1-\langle u,v\rangle^2}
            \leq
            \bigl|P_{u^\perp}A\bigr|_{d-1}
            \bigl|P_{v^\perp}A\bigr|_{d-1}.
            \]
        \item Monotonicity of the projection-volume ratio: Let $A, B \subseteq \R^d$ be zonoids, where $A$ is full-dimensional, and $u \in S^{d-1}$ a direction, then
        \[
            \frac{|A+B|}
                 {\bigl|P_{u^\perp}(A+B)\bigr|_{d-1}}
            \geq
            \frac{|A|}
                 {\bigl|P_{u^\perp}A\bigr|_{d-1}}.
        \]
    
        \item Monotonicity of the volume-to-surface-area ratio: Let $A, B \subseteq \R^d$ be zonoids, where $A$ is full-dimensional, then
        \[
            \frac{|A+B|}{|\partial(A+B)|_{d-1}}
            \geq
            \frac{|A|}{|\partial A|_{d-1}}.
        \]
    \end{itemize}
\end{corollary}

\section{Acknowledgments}
GPT-5.6 Pro was used during the development of this paper to aid with literature searches and with exploratory volume computations of 2-modular zonotopes. All computations were independently verified by the author.

{\small
\printbibliography
}

\end{document}